\documentclass[a4paper,11pt,reqno]{amsart}

\usepackage{latexsym,dsfont,amssymb,amsmath,amsthm}

\usepackage[all]{xy}

\chardef\bslash=`\\ 

\numberwithin{equation}{section}

\newtheorem{theorem}{Theorem}[section]
\newtheorem{corollary}[theorem]{Corollary}
\newtheorem{lemma}[theorem]{Lemma}
\newtheorem{proposition}[theorem]{Proposition}

\theoremstyle{remark}
\newtheorem{remark}[theorem]{Remark}

\theoremstyle{definition}

\newcommand\bp{\begin{proof}}
\newcommand\ep{\end{proof}}

\newcommand{\C}{\mathbb C}
\newcommand{\N}{\mathbb N}
\newcommand{\Q}{\mathbb Q}

\newcommand\M{\mathfrak M}
\newcommand\NN{\mathfrak N}

\newcommand\LR{L(R)}
\newcommand\CR{\C[R]}

\newcommand\eps{\varepsilon}

\newcommand\image{\operatorname{im}}
\newcommand\Mn{\operatorname{Mat}_n}
\newcommand\Mod{\operatorname{-Mod}}
\newcommand\Tor{\operatorname{Tor}}

\newcommand\enu[1]{\smallskip\newline\noindent{\rm(#1)}}

\begin{document}

\title{On the definition of $L^2$-Betti numbers of equivalence relations}

\author[S. Neshveyev]{Sergey Neshveyev}
\address{Department of Mathematics, University of Oslo,
P.O. Box 1053 Blindern, N-0316 Oslo, Norway.}
\email{sergeyn@math.uio.no}

\author[S. Rustad]{Simen Rustad}
\address{Department of Mathematics, University of Oslo,
P.O. Box 1053 Blindern, N-0316 Oslo, Norway.}
\email{simenru@math.uio.no}

\begin{abstract}
We show that the $L^2$-Betti numbers of equivalence relations defined
by R. Sauer coincide with those defined by D. Gaboriau.
\end{abstract}

\date{June 3, 2008}

\maketitle

\section*{Introduction}

The notion of $L^2$-Betti numbers of countable standard equivalence
relations was introduced in a celebrated paper of
Gaboriau~\cite{Ga1}. A few years later a different definition was
given by Sauer~\cite{S}. While Gaboriau's construction was motivated
by Cheeger and Gromov's definition of $L^2$-Betti numbers of
discrete groups~\cite{CG}, Sauer was inspired by the algebraic
framework developed by L\"uck~\cite{L}. Each definition has its own
advantages. E.g. the proof of the theorem of Gaboriau that orbit
equivalent groups have the same $L^2$-Betti numbers is quite short
and transparent in his setting. On the other hand, the computational
power of homological algebra is better accessible through Sauer's
definition, see~\cite{ST}. The two approaches are equivalent for
equivalence relations generated by free actions of discrete
groups~\cite{S}. The aim of this note is to show that they are
equivalent in general.

\bigskip

\section{Dimension theory and homological algebra}

Let $M$ be a finite von Neumann algebra with a fixed faithful normal
tracial state $\tau$. For a finitely generated projective $M$-module
$P\cong M^np$, where $p\in\Mn(M)=M\otimes\Mn(\C)$ is a projection,
its dimension is defined by
$\dim_MP=(\tau\otimes\operatorname{Tr})(p)$. L\"uck extended the
dimension function to all $M$-modules by letting
$$
\dim_MQ=\sup\{\dim_MP\mid P\subset Q\hbox{ is
projective}\}\in[0,+\infty],
$$
see \cite{L}. The most important properties of $\dim_M$ are
additivity and cofinality. Together they imply that if $Q$ is an
inductive limit of modules $Q_i$ with $\dim_MQ_i<\infty$ then
$$
\dim_MQ=\lim_i\lim_j\dim_M\operatorname{im}(Q_i\to Q_j).
$$

A morphism $h\colon Q_1\to Q_2$ of $M$-modules is called a
$\dim_M$-isomorphism if both $\ker h$ and $\operatorname{coker}h$
have dimension zero. By localizing the category $M\Mod$ of
$M$-modules by the subcategory of zero-dimensional modules one can
deal with $\dim_M$-isomorphisms as with usual isomorphisms. What
makes life even better is that the localized category can be
embedded back into the category of $M$-modules using the functor of
rank completion introduced by Thom~\cite{Th}. The definition of this
functor is motivated by the following criterion~\cite{S}: an
$M$-module $Q$ has dimension zero if and only if for any $\xi\in Q$
and $\eps>0$ there exists a projection $p\in M$ such that $p\xi=\xi$
and $\tau(p)<\eps$. Now for $Q\in M\Mod$ and $\xi\in Q$ define
$$
[\xi]_M=\inf\{\tau(p)\mid p\hbox{ is a projection in }M,\
p\xi=\xi\}.
$$
Then $d_M(\xi,\zeta):=[\xi-\zeta]_M$ is a pseudometric on $Q$.
Denote by $c_M(Q)$ the completion of $Q$ in this pseudometric, that
is, the quotient of the module of Cauchy sequences by the submodule
of sequences converging to zero. Any $M$-module map $h\colon Q_1\to
Q_2$ is a contraction in the pseudometric $d_M$, hence it defines a
morphism $c_M(h)\colon c_M(Q_1)\to c_M(Q_2)$. Therefore $c_M$ is a
functor $M\Mod\to M\Mod$, called the functor of rank
completion~\cite{Th}. Notice that although $d_M$ depends on the
choice of the trace, the corresponding uniform structure does not,
so the functor~$c_M$ does not depend on the choice of the trace
either.

\begin{lemma}[\cite{Th}] \label{lThom}
We have:\enu{i} for any $Q\in M\Mod$ the completion map $Q\to
c_M(Q)$ is a $\dim_M$-isomorphism; \enu{ii} $\dim_MQ=0$ if and only
if $c_M(Q)=0$; more generally, a morphism $h\colon Q_1\to Q_2$ is a
$\dim_M$-isomorphism if and only if $c_M(h)\colon c_M(Q_1)\to
c_M(Q_2)$ is an isomorphism; \enu{iii} the functor $c_M$ is exact.
\end{lemma}

We remark that our setting is not the same as that studied by
Thom~\cite{Th}. There, he considers $M$-bimodules $Q$ and defines
$$
[\xi]=\inf\{\tau(p)+\tau(q)\mid p\xi q = \xi\}.
$$
However, all the proofs work equally well if we rather than
$M$-bimodules consider $M$-$N$-bimodules. Our situation then
corresponds to the case when $N$ consists of the scalars.
Furthermore, part (i) and the first part of (ii) in the above lemma
follow immediately by definition and the criterion of zero
dimensionality, while the second part of (ii) then follows from
exactness. Thus the only statement in Lemma~\ref{lThom} which
requires a proof is part (iii).
See \cite[Lemma~2.6]{Th} for details.

If $Q\subset P$ is dense in the pseudometric $d_M$ then we say that
$Q$ is $M$-dense in $P$. If $Q=c_M(Q)$, we say that $Q$ is
$M$-complete.

\medskip

For a pair of algebras $N\subset M$ we shall always assume that both
the algebras and the embedding are unital. Furthermore, if $N$ and
$M$ are finite von Neumann algebras then we shall assume that the trace on $N$ is the restriction of the trace on $M$.

\begin{lemma} \label{lcomp}
Assume $N\subset\M$ is a pair of algebras such that $N$ is a finite
von Neumann algebra. Then the following conditions are equivalent:
\enu{i} for any $Q\in\M\Mod$ and $m\in\M$ the map $Q\to Q$,
$\xi\mapsto m\xi$, is uniformly continuous with respect to the
pseudometric $d_N$; \enu{ii} for any $m\in\M$ and $\eps>0$ there
exists $\delta>0$ such that if $p\in N$ is a projection with
$\tau(p)<\delta$ then $[mp]_N<\eps$; \enu{iii} if $Q\in N\Mod$ is
such that $\dim_NQ=0$ then $\dim_N(\M\otimes_NQ)=0$.
\end{lemma}

\bp Applying (i) to $Q=\M$ we immediately get (ii), so
(i)$\Rightarrow$(ii). Conversely, assume (ii) is satisfied. Let
$Q\in\M\Mod$, $m\in\M$ and $\eps>0$. Choose $\delta>0$ as in (ii).
Then if $\xi\in Q$ and $[\xi]_N<\delta$, we can find a projection
$p\in N$ with $p\xi=\xi$ and $\tau(p)<\delta$, and get
$$
[m\xi]_N=[mp\xi]_N\le [mp]_N<\eps.
$$
Furthermore, a similar computation shows
that $[m\otimes\xi]_N<\eps$. Thus (ii)$\Rightarrow$(i) and
(ii)$\Rightarrow$(iii).

It remains to show that (iii) implies (ii). Assume (ii) is not true.
Then there exist $m\in\M$, $\eps>0$ and a sequence of projections
$p_n\in N$ such that $\tau(p_n)\to0$ but $[mp_n]_N\ge\eps$ for all $n$.
Passing to a subsequence we may assume that $\sum_n\tau(p_n)<\infty$.
Consider the $N$-module $Q=(\prod_n Np_n)/(\oplus_n Np_n)$. Observe
that if $\xi=(\xi_n)_{n\ge k}\in\prod_{n\ge k}Np_n$ then
$$
[\xi]_N\le\sum_{n=k}^\infty[\xi_n]_N\le \sum_{n=k}^\infty\tau(p_n).
$$
This implies that $\dim_NQ=0$. So assuming (iii) we have
$\dim_N(\M\otimes_N Q)=0$. In particular, by considering the image of
$\xi:=m\otimes(p_n)_n\in\M\otimes_N(\prod_nNp_n)$ in $\M\otimes_NQ$, we
can find a projection $p\in N$ such that $\tau(p)<\eps$ and $\xi-p\xi$
lies in the image of $\M\otimes_N(\oplus_nNp_n)$. By considering the
projection $\prod_nNp_n\to Np_k$ onto the $k$-th factor we conclude
that $m\otimes p_k=pm\otimes p_k\in \M\otimes_N Np_k$ for all $k$
sufficiently large. As $\M\otimes_N Np_k=\M p_k$, this shows that
$[mp_k]_N=[pmp_k]_N\le\tau(p)<\eps$ for all $k$ big enough. This
contradicts our choice of the sequence $\{p_n\}_n$. The contradiction
shows that (iii)$\Rightarrow$(ii). \ep

Under the equivalent conditions of the above lemma, the
multiplication by $m\in\M$ on $Q\in\M\Mod$ extends by continuity to
a map on $c_N(Q)$. Therefore the functor $c_N$ of rank completion on
$N\Mod$ defines a functor $\M\Mod\to\M\Mod$ which we denote,
slightly abusing notation, by the same symbol $c_N$. It follows from
\cite{Th} that if $P$ is a projective $\M$-module then $c_N(P)$ is
projective in the category $\M\Mod_c$ of $N$-complete $\M$-modules,
so that if $Q\in\M\Mod_c$ then any surjective morphism $h\colon Q\to
c_N(P)$ has a right inverse. Indeed, the completion morphism $P\to
c_N(P)$ lifts to a morphism $s\colon P\to Q$ by projectivity of $P$,
and then $c_N(s)\colon c_N(P)\to c_N(Q)=Q$ is a right inverse of
$h$. It follows that any exact sequence of $\M$-modules of the form
$0\leftarrow c_N(P_0)\leftarrow c_N(P_1)\leftarrow\dots$, where
the~$P_n$ are projective $\M$-modules, is split-exact.

\begin{lemma} \label{liso}
Let $N\subset\M\subset M$ be a triple of algebras such that $N$ and $M$
are finite von Neumann algebras and the pair $N\subset\M$ satisfies the
equivalent conditions of Lemma~\ref{lcomp}. Then any morphism $Q_1\to
Q_2$ of $\M$-modules which is a $\dim_N$-isomorphism induces a
$\dim_M$-isomorphism $\Tor^\M_n(M,Q_1)\to\Tor^\M_n(M,Q_2)$ for all
$n\ge0$.
\end{lemma}

\bp This is proved in~\cite{S} and in a different form in \cite{Th}. We
shall nevertheless sketch a proof for the reader's
convenience.

\smallskip

Consider the case $n=0$. It suffices to show that for any $\M$-module
$Q$ the completion map $Q\to c_N(Q)$ induces a $\dim_M$-isomorphism
$M\otimes_\M Q\to M\otimes_\M c_N(Q)$. Since $[mp]_M\le\tau(p)$ for any
$m\in M$ and any projection $p\in M\supset N$, we have
$[m\otimes\xi]_M\le[\xi]_N$. Hence the image of $M\otimes_\M Q$ is
$M$-dense in $M\otimes_\M c_N(Q)$, and we get a surjective morphism
\begin{equation} \label{ecomp}
c_M(M\otimes_\M Q)\to c_M(M\otimes_\M c_N(Q)).
\end{equation}
On the other hand, if $\{\xi_n^k\}_n$, $k=1,\dots,l$, are Cauchy
sequences in $Q$, then for any $m_1,\dots,m_l\in M$ the sequence
$\xi_n=\sum^l_{k=1}m_k\otimes\xi_n^k$ is Cauchy in $M\otimes_\M Q$
(in the pseudometric $d_M$), so it defines an element of
$c_M(M\otimes_\M Q)$. Moreover, if $[\xi_n^k]_N\to0$ as $n\to\infty$
for all $k$ then $[\xi_n]_M\to0$. Since $c_N(Q)$ is the quotient of
the module of Cauchy sequences by the submodule of sequences
converging to zero, we therefore get a well-defined map
$M\otimes_\M c_N(Q)\to c_M(M\otimes_\M Q)$ with $M$-dense image, and
hence a surjective morphism
$$
c_M(M\otimes_\M c_N(Q))\to c_M(M\otimes_\M Q).
$$
Clearly, it is the inverse of \eqref{ecomp}.

\smallskip

Turning to the general case, consider first an $\M$-module $Q$ such
that $\dim_NQ=0$. We have to show that $\dim_M\Tor^\M_n(M,Q)=0$ for all
$n\ge0$. Consider a projective resolution $0\leftarrow Q\leftarrow
P_\bullet$. Since $c_N(Q)=0$, the complex $0\leftarrow c_N(P_\bullet)$
is exact. By the remark before the lemma, it is therefore split-exact.
It follows that $0\leftarrow M\otimes_\M c_N(P_\bullet)$ is exact. On
the other hand, by the first part of the proof this complex is
$\dim_M$-isomorphic to the complex $0\leftarrow M\otimes_\M P_\bullet$.
Since $\Tor^\M_n(M,Q)\cong H_n(M\otimes_\M P_\bullet)$, we conclude that
$\dim_M\Tor^\M_n(M,Q)=0$ for all $n\ge0$.

Finally, for an arbitrary morphism $h\colon Q_1\to Q_2$ of
$\M$-modules which is a $\dim_N$-isomorphism, consider the short
exact sequences
$$
0\to\ker h\to Q_1\to\operatorname{im} h\to0\ \ \hbox{and}\ \
0\to\operatorname{im} h\to Q_2\to\operatorname{coker} h\to0
$$
and the corresponding long exact sequences of $\Tor$-groups. Since
$\dim_M\Tor^\M_n(M,\ker h)=0$ and
$\dim_M\Tor^\M_n(M,\operatorname{coker} h)=0$ for all $n$, we then see
that $\Tor_n^\M(M,Q_1)$ and $\Tor_n^\M(M,Q_2)$ are $\dim_M$-isomorphic.
\ep


\begin{lemma} \label{lres}
Let $N\subset\M\subset M$ be a triple of algebras such that $N$ and $M$
are finite von Neumann algebras. Assume $0\leftarrow Q\leftarrow
P_\bullet$ is a resolution of an $\M$-module $Q$ such that
$\dim_M\Tor^\M_n(M,P_k)=0$ for all $n\ge1$ and $k\ge0$. Then
$$
\dim_M\Tor^\M_n(M,Q)=\dim_M H_n(M\otimes_\M P_\bullet)\ \ \hbox{for all}\ \
n\ge0.
$$
In particular, if the pair $N\subset\M$ satisfies the equivalent
conditions of Lemma~\ref{lcomp}, then to compute
$\dim_M\Tor^\M_\bullet(M,Q)$ one can use any resolution of $Q$ by
$\M$-modules that contain $N$-dense projective $\M$-submodules.
\end{lemma}

\bp Consider the functor $F=c_M(M\otimes_\M\, \cdot)\colon\M\Mod\to
M\Mod$. Since $c_M$ is exact, for the derived functors of $F$ we
have $L_nF=c_M\circ
L_n(M\otimes_\M\,\cdot)=c_M(\Tor_n^\M(M,\cdot))$. Therefore the
assumption of the lemma says that $0\leftarrow Q\leftarrow
P_\bullet$ is an $F$-acyclic resolution of $Q$. Hence
$$
c_M(\Tor_n^\M(M,Q))=L_nF(Q)\cong H_n(F(P_\bullet))
=H_n(c_M(M\otimes_\M P_\bullet))\cong c_M(H_n(M\otimes_\M P_\bullet)),
$$
which proves the first part of the lemma.

The second part follows from Lemma~\ref{liso}, since if an $\M$-module
$P$ contains an $N$-dense projective submodule $\tilde P$, then by that
lemma the modules $\Tor_n^\M(M,P)$ and $\Tor_n^\M(M,\tilde P)=0$ (for
$n\ge1$) are $\dim_M$-isomorphic. \ep

The following remarks will not be used later, but may be of independent
interest.

\begin{remark} \mbox{\ }

\noindent (i) Lemma~\ref{liso} can be strengthened as follows.
Assume $M$ is a finite von Neumann algebra, $N\subset\M$ is a pair
satisfying the equivalent conditions of Lemma~\ref{lcomp}, and $R$
is an $M$-$\M$-bimodule satisfying the following equivalent (by an
analogue of Lemma~\ref{lcomp}) conditions:
\begin{itemize}
\item[--] for any $r\in R$ and $\eps>0$ there exists $\delta>0$ such
that if $p\in N$ is a projection with $\tau(p)<\delta$ then
$[rp]_M<\eps$; \item[--] if $Q\in N\Mod$ is such that $\dim_NQ=0$
then $\dim_M(R\otimes_NQ)=0$.
\end{itemize}
Then any morphism $Q_1\to Q_2$ of $\M$-modules which is a
$\dim_N$-isomorphism induces a $\dim_M$-isomorphism
$\Tor^\M_n(R,Q_1)\to\Tor^\M_n(R,Q_2)$ for all $n\ge0$. The proof is
essentially the same as above. This is \cite[Lemma 4.10]{S}, but we
see that the flatness assumption there is not needed.

\medskip\noindent
(ii) Lemma~\ref{lres} provides a mildly alternative route to
\cite[Theorem 4.11]{S}, which is a key point in Sauer's approach to
Gaboriau's theorem on the $L^2$-Betti numbers of orbit equivalent
groups. Namely, assume $N\subset\NN\subset\M\subset M$ is a quadruple
of algebras such that $N$ and $M$ are finite von Neumann algebras,
$\NN$ is $N$-dense in $\M$ and the pair $N\subset\M$ satisfies the
equivalent conditions of Lemma~\ref{lcomp}. Then
$$
\dim_M\Tor^\M_n(M,Q)=\dim_M\Tor^\NN_n(M,Q)
$$
for any $\M$-module $Q$ and all $n\ge0$. Indeed, to compute
$\Tor^\M_n(M,Q)$ we use a resolution $0\leftarrow Q\leftarrow
P_\bullet$ of $Q$ by free $\M$-modules. Since $\NN$ is $N$-dense in
$\M$, by Lemma~\ref{lres} this resolution can also be used to
compute $\dim_M\Tor^\NN_n(M,Q)$. Thus we only need to check that the
canonical map $M\otimes_\NN P_\bullet\to M\otimes_\M P_\bullet$ is a
$\dim_M$-isomorphism. Since $M\otimes_\NN P=M\otimes_\M(\M\otimes_\NN
P)$, by Lemma~\ref{liso} it is enough to check that the map
$h\colon\M\otimes_\NN P\to P$, $m\otimes\xi\mapsto m\xi$, is a
$\dim_N$-isomorphism for any $\M$-module $P$. But this is clear,
since the $N$-module map $P\to\M\otimes_\NN P$, $\xi\mapsto
1\otimes\xi$, is a right inverse to $h$ and has $N$-dense image by
virtue of density of $\NN$ in $\M$.
\end{remark}

\bigskip

\section{$L^2$-Betti numbers}

Let $X$ be a standard Borel space, $R\subset X\times X$ a countable
Borel equivalence relation on $X$ preserving a probability measure
$\mu$. The measure $\mu$ will usually be omitted in our notation,
e.g. we write $L^\infty(X)$ instead of $L^\infty(X,\mu)$. As usual
denote by $[R]$ the group of invertible Borel transformations of $X$
with graphs in $R$.

A standard fiber space over $X$ is a standard Borel space $U$
together with a Borel map $\pi\colon U\to X$ with at most countable
fibers. There is then a natural measure $\nu_U$ on $U$ given by
\[
\nu_U(C)=\int_X \#(\pi^{-1}(x)\cap C)d\mu(x).
\]
The example that we will be the most concerned with in the following
is that where $U=R$ and $\pi$ is either $\pi_l$ or $\pi_r$, the
projections onto the first and second coordinates respectively.
Since $\mu$ is invariant, the induced measures on $R$ are the same,
denoted simply by $\nu$.

Given two standard fiber spaces over $X$, $(U,\pi)$ and $(V,\pi')$,
their fiber product is
\[
U*V = \{(u,v)\in U\times V \mid \pi(u)=\pi'(v)\},
\]
which is again a standard fiber space.

A left $R$-action on a standard fiber space $U$ over $X$ is a Borel
map $(R,\pi_r)*U\to U$ denoted $((x,y),u)\mapsto (x,y)u$, where
$y=\pi(u)$, satisfying
\[
(x,y)((y,z)u)=(x,z)u,\qquad (z,z)u=u
\]
whenever this makes sense. This implies that $\pi((y,z)u)=y$, and
that $(x,y)$ is a bijection between $\pi^{-1}(y)$ and $\pi^{-1}(x)$.

Consider the subspace $\CR$ of $L^\infty(R,\nu)$ consisting of
functions that are supported on finitely many graphs of
$\phi\in[R]$. Equivalently, a function $f\in L^\infty(R,\nu)$
belongs to $\CR$ if
\[
x\mapsto \#\{y\mid f(x,y)\neq 0\} + \#\{y\mid f(y,x)\neq 0\}
\]
is in $L^\infty(X)$. Then $\CR$ is an involutive algebra with
product
\[
(fg)(x,z)=\sum_{y\sim x}f(x,y)g(y,z)
\]
and involution $f^*(x,y)=\overline{f(y,x)}$.

If $(U,\pi)$ is a standard fiber space over $X$, denote by
$\Gamma(U)$ the space of Borel functions $f$ on $U$, considered
modulo sets of $\nu_U$-measure zero, such that the support of
$f|_{\pi^{-1}(x)}$ is finite for a.e. $x\in X$. Furthermore, denote
by $\Gamma^{b}(U)$ the space of functions $f\in \Gamma(U)\cap
L^\infty(U,\nu_U)$ such that
$$
x\mapsto \#\{u\in\pi^{-1}(x)\mid f(u)\neq 0\}
$$
is in $L^\infty(X)$. We shall also denote the space $L^2(U,d\nu_U)$
by $\Gamma^{(2)}(U)$. If $R$ acts on $U$ then all three spaces
$\Gamma(U)$, $\Gamma^b(U)$ and $\Gamma^{(2)}(U)$ are $\CR$-modules
in a natural way. In particular, if $(U,\pi)=(R,\pi_l)$, we get a
representation of $\CR$ on $L^2(R,d\nu)$, and we let $\LR$ be the
von Neumann algebra generated by $\CR$ in this representation. The
characteristic function $\chi_\Delta$ of the diagonal $\Delta\subset
R$ is a cyclic and separating vector for $\LR$ defining a normal
tracial state $\tau$ on $\LR$, so that $L^2(R,d\nu)=L^2(\LR,\tau)$.
Note also that for $U=X$ we have $\Gamma^b(U)=L^\infty(X)$,
$\Gamma^{(2)}(U)=L^2(X,d\mu)$ and $\Gamma(U)=M(X)$, the space of
measurable functions on $X$. In particular, $L^\infty(X)$,
$L^2(X,d\mu)$ and $M(X)$ are left $\CR$-modules.

The results of the previous section will be applied to the triple
$L^\infty(X)\subset\CR\subset \LR$, where $L^\infty(X)$ is
identified with $L^\infty(\Delta,\nu)$. The equivalent conditions of
Lemma~\ref{lcomp} are satisfied. Indeed, if $f\in\CR$ is 
supported on the graph of $\phi\in[R]$ then
$f\chi_Z=\chi_{\phi^{-1}(Z)}f$ for any Borel $Z\subset X$, so that
$$
[f\chi_Z]_{\LR}=[\chi_{\phi^{-1}(Z)}f]_{\LR}\le\mu(\phi^{-1}(Z))
=\mu(Z)=\tau(\chi_Z),
$$
and thus condition (ii) in Lemma~\ref{lcomp} is satisfied for $m=f$
with $\delta=\eps$.

For a general standard fiber space $U$ with an $R$-action the
$\CR$-module structure on $\Gamma^{(2)}(U)$ does not extend to an
action of $\LR$. But it extends for the following class of spaces.
An $R$-action on~$U$ is called discrete if there is a Borel
fundamental domain, that is, if there is a Borel set $F\subset U$
intersecting each $R$-orbit once and only once. For the case
$(U,\pi)=(R,\pi_l)$, the diagonal $\Delta$ is a fundamental domain for
the standard $R$-action. For general discrete $R$-spaces $U$, by
choosing sections of $F\to X$ we can embed~$U$ into
$\bigsqcup^\infty_{n=1}R$, see \cite[Lemma~2.3]{Ga1}, that is, any
discrete $R$-space is $R$-equivariantly isomorphic to
$\bigsqcup^N_{n=1}R\Delta(X_n)$, where $N\in\N\cup\{+\infty\}$, the
$X_n$ are Borel subsets of~$X$, $\Delta(X_n)=\{(x,x)\mid x\in X_n\}$
and therefore $R\Delta(X_n)=\{(y,x)\mid x\in X_n,\ y\sim x\}$. In
particular, the $\CR$-module $\Gamma^{(2)}(U)$ is isomorphic to
$\oplus_nL^2(\LR)\chi_{X_n}$, and hence the action of $\CR$ on it
extends to an action of $\LR$.

If $U$ is a discrete $R$-space and $F\subset U$ is a Borel
fundamental domain, then we write $\nu_U(R\backslash U)$
for~$\nu_U(F)$. If $U\cong\bigsqcup_nR\Delta(X_n)$ then
$\nu_U(R\backslash U)=\sum_n\mu(X_n)$. Since
$\dim_{\LR}L^2(\LR)\chi_Z=\tau(\chi_Z)=\mu(Z)$, we also get
$$
\nu_U(R\backslash U)=\dim_{\LR}\Gamma^{(2)}(U).
$$

If $U$ and $V$  are standard fiber spaces with $R$-action, then
$U*V$ is again a standard fiber space with diagonal action of $R$.
Furthermore, if $F$ is a Borel fundamental domain for $U$, then
$F*V$ is a Borel fundamental domain for $U*V$.

\medskip

A simplicial $R$-complex $\Sigma$ consists of a discrete $R$-space
$\Sigma^0$ and Borel sets $\Sigma^1,\Sigma^2,\ldots$ with
\[
\Sigma^n\subset\overbrace{\Sigma^0*\cdots *\Sigma^0}^{n+1}
\]
satisfying, for $n>0$, \enu{i} $R\Sigma^n=\Sigma^n$; \enu{ii} if
$(v_0,\ldots,v_n)\in\Sigma^n$ then
$(v_{\sigma(0)},\ldots,v_{\sigma(n)})\notin\Sigma^n$ for any
nontrivial permutation $\sigma$; \enu{iii} if
$(v_0,\ldots,v_n)\in\Sigma^n$ then for all $i=0,\dots,n$ a
permutation of $(v_0,\ldots,\hat v_i,\ldots,v_n)$ is
in~$\Sigma^{n-1}$.
\smallskip

Note that this definition is slightly different from that
in~\cite{Ga1}, as we prefer to fix an order on the vertices
of every simplex; in particular, our simplices are oriented.

Given a simplicial $R$-complex $\Sigma$, we may associate to it a
field of simplicial complexes $\Sigma_x$ by letting $\Sigma_x^n$ be
the fiber of $\Sigma^n\to X$ over $x$. One says that $\Sigma$ is
$n$-dimensional, contractible, and so on, if these properties hold
for $\Sigma_x$ for $\mu$-a.e. $x\in X$.

For a simplicial $R$-complex $\Sigma$ we put
$$
C_n^b(\Sigma)=\Gamma^b(\Sigma^n),\ \
C_n(\Sigma)=\Gamma(\Sigma^n),\ \
C_n^{(2)}(\Sigma)=\Gamma^{(2)}(\Sigma^n).
$$
The boundary operators $\partial_{n,x}\colon C_n(\Sigma_x)\to
C_{n-1}(\Sigma_x)$ define a $\CR$-module map $\partial_n\colon
C_n(\Sigma)\to C_{n-1}(\Sigma)$. It maps $C_n^b(\Sigma)$ into
$C_{n-1}^b(\Sigma)$.

A simplicial $R$-complex $\Sigma$ is called uniformly locally
bounded (ULB) if there is an integer $m$ such that every vertex of
$\Sigma_x$ is contained in no more than $m$ simplices for almost
every $x\in X$, and if furthermore $\Sigma^0$ has a fundamental
domain of finite measure. The first condition guarantees that the
boundary operators $\partial_{n,x}$ define a bounded $\LR$-module map
$\partial_n\colon C_n^{(2)}(\Sigma)\to C_{n-1}^{(2)}(\Sigma)$. The
second condition is equivalent to
$\dim_{\LR}C_0^{(2)}(\Sigma)<\infty$. The two conditions together
imply that $\Sigma^n$  has a fundamental domain of finite measure
for any $n$, that is, $\dim_{\LR}C_n^{(2)}(\Sigma)<\infty$.

\smallskip

For a ULB simplicial $R$-complex $\Sigma$, its $n$-th reduced
$L^2$-homology is defined by
$$
\bar H_n^{(2)}(\Sigma,R)=\ker(\partial_n\colon
C_n^{(2)}(\Sigma)\to C_{n-1}^{(2)}(\Sigma))/
\overline{\operatorname{im}(\partial_{n+1}\colon
C_{n+1}^{(2)}(\Sigma)\to C_n^{(2)}(\Sigma))},
$$
and then its $n$-th $L^2$-Betti number is defined by
$$
\beta^{(2)}_n(\Sigma,R)=\dim_{\LR}\bar H_n^{(2)}(\Sigma,R).
$$
For a general simplicial $R$-complex $\Sigma$ consider an exhaustion
$\{\Sigma_i\}^\infty_{i=1}$ of $\Sigma$ by ULB complexes, that is,
$\Sigma_i^n\subset\Sigma_{i+1}^n\subset\Sigma^n$ and
$\cup_i\Sigma^n_{i,x}=\Sigma^n_x$ for a.e. $x\in X$ and all $n\ge0$.
Then define
\[
\beta_n^{(2)}(\Sigma,R)=\lim_i\lim_j\dim_{\LR}\overline{ \image(\bar
H_n^{(2)}(\Sigma_i,R)\to \bar H_n^{(2)}(\Sigma_j,R))}.
\]
It is shown in~\cite{Ga1} that $\beta_n^{(2)}(\Sigma,R)$ does not
depend on the choice of exhaustion. This will also follow from the
proof of the next result.

\begin{proposition} \label{pLuck}
For any simplicial $R$-complex $\Sigma$ we have
$$
\beta_n^{(2)}(\Sigma,R)=\dim_{\LR}H_n(\LR\otimes_{\CR}C^b_\bullet(\Sigma))
=\dim_{\LR}H_n(\LR\otimes_{\CR}C_\bullet(\Sigma)).
$$
\end{proposition}

This is analogous to the fact that if a discrete group $G$ acts
freely on a simplicial complex $\Sigma$ then
$\beta^{(2)}_n(\Sigma,G)
=\dim_{L(G)}H_n(L(G)\otimes_{\C[G]}C_\bullet(\Sigma))$,
see~\cite{L}, and the proof is similar, although one needs a bit
more care in dealing with different chain spaces. For the proof we
will need the following lemma.

\begin{lemma} \label{ldense}
Let $U$ be a discrete $R$-space. Then \enu{i} $\Gamma(U)$ is the
$L^\infty(X)$-completion of a projective $\CR$-module; \enu{ii} if
$\nu_U(R\backslash U)<\infty$, then the map
$\LR\otimes_{\CR}\Gamma^b(U)\to \Gamma^{(2)}(U)$, $m\otimes
\xi\mapsto m\xi$, is a $\dim_{\LR}$-isomorphism.
\end{lemma}

\begin{proof}
We may assume that $U=\bigsqcup^N_{n=1}R\Delta(X_n)$.
Consider the projective submodule
\[
P=\bigoplus_{n=1}^N\CR\chi_{X_n}
\]
of $\Gamma^b(U)$. We claim that it is $L^\infty(X)$-dense in
$\Gamma(U)$. Indeed, let $f\in\Gamma(U)$. For $m\in\N$ consider the
set
$$
Y_m=\{x\in X\mid \operatorname{supp}
f|_{\pi^{-1}(x)}\subset\cup^m_{n=1}R\Delta(X_n)\}.
$$
Then $\{Y_m\}_m$ is an increasing sequence of Borel sets with union a
subset of $X$ of full measure. Thus $\chi_{Y_m}f\to f$ in the metric
$d_{L^\infty(X)}$ as $m\to\infty$. Furthermore, $\chi_{Y_m}f$ is
supported on $\cup^m_{n=1}R\Delta(X_n)$. Therefore we may assume
that $N$ is finite. But then it suffices to show that $\CR$ is
$L^\infty(X)$-dense in $\Gamma(R)$.

Choose a sequence of transformations $\phi_n\in[R]$ such that $R$ is
the union of the graphs of $\phi_n$. For $f\in\Gamma(R)$ and
$m\in\N$ consider the set
$$
Z_m=\{x\in X\mid |f(x,y)|\le m \hbox{ for all }y\sim x,\
\operatorname{supp}f(x,\cdot)\subset\{\phi_1(x),\dots,\phi_m(x)\}\}.
$$
Then $\chi_{Z_m}f\to f$ in the metric $d_{L^\infty(X)}$ as $m\to\infty$,
and $\chi_{Z_m}f\in\CR$. This finishes the proof of density of $P$
in $\Gamma(U)$.

To finish the proof of (i) it remains to check that $\Gamma(U)$ is
$L^\infty(X)$-complete. For this one can observe that if $Q$ is an
$M$-module for a finite von Neumann algebra $M$, then for any Cauchy
sequence in $Q$ one can choose a subsequence $\{\xi_n\}_n$ for which there is an increasing sequence of projections
$p_n\in M$ converging  strongly to the unit such that
$p_n\xi_n=p_n\xi_m$ for all $m\ge n$. But if we have a sequence of
this form in $\Gamma(U)$, it obviously converges to an element of
$\Gamma(U)$.

\smallskip

Turning to (ii), we have
$$
\Gamma^{(2)}(U)=\bigoplus_{n=1}^NL^2(\LR)\chi_{X_n}
\ \ \hbox{(Hilbert space direct sum)}.
$$
We claim that $\LR\otimes_{\CR}P\to \Gamma^{(2)}(U)$, $m\otimes
\xi\mapsto m\xi$, is a $\dim_{\LR}$-isomorphism. Indeed,
$$
\LR\otimes_{\CR}P=\bigoplus_n\LR\chi_{X_n}.
$$
Since $\LR$ is $\LR$-dense in $L^2(\LR)$, we see that
$\LR\otimes_{\CR}P$ is $\LR$-dense in the algebraic direct sum of
$L^2(\LR)\chi_{X_n}$, $n=1,\dots,N$. On the other hand, since
$\sum_n\mu(X_n)<\infty$ by assumption, the algebraic direct sum is
$\LR$-dense in the Hilbert space direct sum (because if $\xi\in
L^2(\LR)p$ for a projection $p\in \LR$ then
$[\xi]_{\LR}\le\tau(p)$). This proves our claim. Since $P$ is
$L^\infty(X)$-dense in $\Gamma^b(U)$ by~(i), by Lemma~\ref{liso} we
conclude that $\LR\otimes_{CR}\Gamma^b(U)\to \Gamma^{(2)}(U)$ is a
$\dim_{\LR}$-isomorphism.
\end{proof}

\bp[Proof of Proposition~\ref{pLuck}] We start with the first
equality. Assume that $\Sigma$ is a ULB simplicial $R$-complex. In
this case the dimension (over $\LR$) of the module
$\operatorname{im}(\partial_{n+1}\colon C_{n+1}^{(2)}(\Sigma)\to
C_n^{(2)}(\Sigma))$ coincides with the dimension of its Hilbert
space closure, since $\partial_{n+1}\partial^*_{n+1}$ maps
$\overline{\image\partial_{n+1}}$ injectively into
$\image\partial_{n+1}$. It follows that the canonical surjection
$$
H_n(C^{(2)}_\bullet(\Sigma))\to \bar
H^{(2)}_n(\Sigma,R)
$$
is a $\dim_{\LR}$-isomorphism. On the other hand, by
Lemma~\ref{ldense} we have a canonical $\dim_{\LR}$-isomorphism
$\LR\otimes_{\CR}C^b_\bullet(\Sigma)\to C^{(2)}_\bullet(\Sigma)$.
Therefore we obtain a canonical $\dim_{\LR}$-isomorphism
$$
H_n(\LR\otimes_{\CR}C^b_\bullet(\Sigma))\to \bar
H^{(2)}_n(\Sigma,R),
$$
which gives the desired result for $\Sigma$.

For a general simplicial $R$-complex $\Sigma$ consider an exhaustion
of $\Sigma$ by ULB $R$-complexes $\Sigma_i$, $i\ge1$. Then by
definition of $\beta^{(n)}_2(\Sigma,R)$, the ULB case and the fact
that the image of $\bar H^{(2)}_n(\Sigma_i,R)$ in $\bar
H^{(2)}_n(\Sigma_j,R)$ (for $j>i$) has the same dimension as its
Hilbert space closure, we can write
$$
\beta_n^{(2)}(\Sigma,R)=\lim_i\lim_j\dim_{\LR}
\image(H_n(\LR\otimes_{\CR}C^b_\bullet(\Sigma_i))\to
H_n(\LR\otimes_{\CR}C^b_\bullet(\Sigma_j))).
$$
Since the inductive limit of
$H_n(\LR\otimes_{\CR}C^b_\bullet(\Sigma_i))$ is isomorphic to
$H_n(\LR\otimes_{\CR}(\cup_iC^b_\bullet(\Sigma_i)))$, by cofinality
and additivity of the dimension function we get
$$
\beta_n^{(2)}(\Sigma,R)=\dim_{\LR}H_n(\LR
\otimes_{\CR}(\cup_iC^b_\bullet(\Sigma_i))).
$$
Next observe that $\cup_iC^b_\bullet(\Sigma_i)$ is
$L^\infty(X)$-dense in $C^b_\bullet(\Sigma)$ (we had a similar
argument in the proof of part (i) of Lemma~\ref{ldense}). By Lemma~\ref{liso} we conclude that
$$
\beta_n^{(2)}(\Sigma,R)=\dim_{\LR}H_n(\LR
\otimes_{\CR}C^b_\bullet(\Sigma)).
$$
The second equality in the statement then holds by the
$L^\infty(X)$-density of $C^b_\bullet(\Sigma)$ in
$C_\bullet(\Sigma)$, which follows from the proof of
Lemma~\ref{ldense}(i). \ep

By a result of Gaboriau~\cite{Ga1}, the numbers
$\beta^{(2)}_n(\Sigma,R)$ are the same for any contractible
simplicial $R$-complex $\Sigma$. This will also follow from the
next theorem, which is our main result.

\begin{theorem} \label{tMain}
If $\Sigma$ is a contractible simplicial $R$-complex then
$$
\beta^{(2)}_n(\Sigma,R)=\dim_{\LR}\Tor^{\CR}_n(\LR,L^\infty(X))\ \
\hbox{for all}\ \ n\ge0.
$$
\end{theorem}

To prove the theorem we need a particular resolution of
$L^\infty(X)$. The obvious candidate is
\begin{equation} \label{ereso}
0\leftarrow L^\infty(X)\xleftarrow{\eps}C_0^b(\Sigma)
\xleftarrow{\partial_1} C_1^b(\Sigma)\xleftarrow{\partial_2}\dots,
\end{equation}
where $\eps(f)(x)=\sum_{u\in\pi^{-1}(x)}f(u)$. It is more
convenient to work with its $L^\infty(X)$-completion, the complex
$$
0\leftarrow M(X)\xleftarrow{\eps}C_0(\Sigma)\xleftarrow{\partial_1}
C_1(\Sigma)\xleftarrow{\partial_2}\dots.
$$
Recall that $M(X)$ denotes the space of measurable functions on $X$.
Fiberwise the above complex is contractible, so to check exactness
it suffices to find homotopies depending measurably on $x\in X$.
This will be done using the following two lemmas.

\begin{lemma} \label{lproj}
Let $V$ be a vector space over $\Q$ of countable dimension. Let
$x\mapsto V_x$ be a field of subspaces of $V$ such that
for all measurable mappings $s\colon X\to V$ the set $\{x\in X \mid
s(x)\in V_x\}$ is measurable. Then there is a field of projections
$x\mapsto p_x$ onto $V_x$ which is measurable in the sense that for
every measurable mapping $s\colon X\to V$ the map $x\mapsto
p_xs(x)\in V$ is measurable.
\end{lemma}

\begin{proof}
Let $\{e_1,e_2,\ldots\}$ be a basis for $V$, and set
$V_n=\operatorname{Span}_\Q\{e_1,\ldots,e_n\}$. We claim that there
exist unique projections $p_x\colon V\to V_x$ such that
\begin{itemize}
\item[(a)] $p_xV_k\subset V_k$ for all $k\ge1$; \item[(b)]
if $V_k\cap V_x\subset V_{k-1}$ for some $k$ then $p_xe_k=0$.
\end{itemize}
To show this we shall prove by induction on $n$ that there exist
unique projections $p_x\colon V_n\to V_n\cap V_x$ satisfying
properties (a) and (b) for $k\le n$. For $n=1$ this is trivial, as
$p_xe_1=\chi_{V_x}(e_1)e_1$ is the only possible option for $p_x$.
Assume by induction that $p_x$ is defined on $V_{n-1}$. We have two
possibilities: either $V_n\cap V_x=V_{n-1}\cap V_x$ or $\dim V_n\cap
V_x=\dim (V_{n-1}\cap V_x)+1$. In the first case the condition
$p_xe_n=0$ completely determines an extension of $p_x$ to $V_n$. In
the second case there exists only one extension, since if
$v\in(V_n\cap V_x)\setminus V_{n-1}$ then $V_n=V_{n-1}\oplus\Q v$
and $V_n\cap V_x=(V_{n-1}\cap V_x) \oplus\Q v$. Thus our claim is
proved.

It remains to show that the field $x\mapsto p_x$ is measurable. For
this it suffices to check that the maps $x\mapsto p_xe_n$ are
measurable. Let $U_1,U_2,\ldots$ be an enumeration of the subspaces
of $V_n$, and let $X_m\subset X$ be the set of $x$ such that
$V_n\cap V_x=U_m$. For any $m$, the set $U_m$ is measurable by
assumption and the vector $p_xe_n$ is the same for all $x\in X_m$ by
uniqueness of $p_x$. Hence $x\mapsto p_xe_n$ is measurable.
\end{proof}

\begin{lemma} \label{lhomo}
Let $V$ be a vector space over $\Q$ of countable dimension. Let
$T_x, p_x, q_x\colon V\to V$ be measurable fields of operators, with
$p_x$ and $q_x$ idempotent. Assume $T_x$ maps $\ker q_x$ bijectively
onto $\image p_x$. Denote by $S_x$ the operator which is zero on
$\ker p_x$ and is the inverse of $T_x\colon \ker q_x\to \image p_x$
on $\image p_x$, so that
$$
T_xS_x=p_x\ \ \hbox{and}\ \ S_xT_x=1-q_x.
$$
Then the field $x\mapsto S_x$ is measurable.
\end{lemma}

\begin{proof}
Let $\{e_1,e_2,\ldots\}$ be a basis for $V$, and enumerate $V$ as
$V=\{v_1,v_2,\ldots\}$. For $i,j\in\N$, put
\[
X_{ij}=\{x\in X : v_i\in \image(1-q_x),\ T_xv_i=p_xe_j\}.
\]
Then the $X_{ij}$ are measurable with $\bigsqcup_{i=1}^\infty
X_{ij}=X$ for all $j\in\N$, and so the field of operators given by
\[
S_xe_j = v_i \ \ \hbox{for}\ \ x\in X_{ij}
\]
is measurable. Furthermore, it clearly has the stated properties.
\end{proof}

We are now ready to prove exactness.

\begin{proposition} \label{pExact}
Let $\Sigma$ be a contractible simplicial $R$-complex. Then
$$
0\leftarrow M(X)\xleftarrow{\eps}C_0(\Sigma)\xleftarrow{\partial_1}
C_1(\Sigma)\xleftarrow{\partial_2}\dots
$$
is contractible as a complex of $L^\infty(X)$-modules.
\end{proposition}

\begin{proof}
Consider first the same sequence with rational coefficients. By
choosing Borel sections of $\Sigma^n\to X$ we can embed $\Sigma^n$
into the trivial fiber space  $X\times\N$ over $X$ and then apply
Lemma~\ref{lproj} to the field of spaces
$x\mapsto\ker\partial_{n,x}\subset C_n(\Sigma_x;\Q)$. Then we get a
field of projections $p_{n,x}\colon C_n(\Sigma_x;\Q)
\to\ker\partial_{n,x}$ which is measurable in the sense that it
determines a well-defined map of $C_n(\Sigma;\Q)$ into itself. By
contractibility of $\Sigma_x$ the map $\partial_{n+1,x}$ is an
isomorphism of $\ker p_{n+1,x}$ onto $\image p_{n,x}$. By
Lemma~\ref{lhomo} we thus get measurable fields of operators
$h_{n,x}=S_x\colon C_n(\Sigma_x;\Q)\to C_{n+1}(\Sigma_x;\Q)$ such
that
\begin{equation} \label{ehomot}
\operatorname{id} = h_{n-1,x}\partial_{n,x} + \partial_{n+1,x}h_{n,x}.
\end{equation}
on $C_n(\Sigma_x;\Q)$ for all $n\ge-1$ (with
$C_{-1}(\Sigma_x;\Q)=\Q$, $\partial_{0,x}=\eps_x$ and $h_{-1,x}\colon \Q
\mapsto C_0(\Sigma_x;\Q)$).

Turning to complex coefficients, extend $h_{n,x}$ to operators
$C_n(\Sigma_x)\to C_{n+1}(\Sigma_x)$ by linearity. These operators
form a measurable field for every $n$, since if $f\in C_n(\Sigma)$
is supported on the image of a section of $\Sigma^n\to X$ then $f$
is an element of $C_n(\Sigma;\Q)$ multiplied with a function
in~$L^\infty(X)$. By \eqref{ehomot} the maps $h_{n,x}$ define the
required homotopy $h_n\colon C_n(\Sigma)\to C_{n+1}(\Sigma)$.
\end{proof}

Notice that the above proposition does not imply that complex
\eqref{ereso} is exact, only that its homology is zero-dimensional
over $L^\infty(X)$.

\bp[Proof of Theorem~\ref{tMain}] Since $L^\infty(X)$ is
$L^\infty(X)$-dense in $M(X)$, by Lemma~\ref{liso} we have
$$
\dim_{\LR}\Tor^{\CR}_n(\LR,L^\infty(X))
=\dim_{\LR}\Tor^{\CR}_n(\LR,M(X)).
$$
To compute the latter numbers, by Lemma~\ref{ldense}(i) and
Lemma~\ref{lres} we can use the resolution of $M(X)$ given by
Proposition~\ref{pExact}. The result follows then from
Proposition~\ref{pLuck}. \ep

Gaboriau~\cite{Ga1} defined the $L^2$-Betti numbers of $R$ by
letting $\beta^{(2)}_n(R)=\beta^{(2)}_n(\Sigma,R)$, where $\Sigma$
is an arbitrary contractible simplicial $R$-complex. By the above
result this definition is equivalent to that of Sauer~\cite{S},
$\beta^{(2)}_n(R)= \dim_{\LR}\Tor^{\CR}_n(\LR,L^\infty(X))$. The
proof also shows the following.

\begin{corollary}[\cite{Ga1}, Theorem~3.13]
If $\Sigma$ is an $n$-connected simplicial $R$-complex, then
\[
\beta_k^{(2)}(\Sigma,R)=\beta_k^{(2)}(R) \ \ \hbox{for}\ \ 0\le k\le
n,\ \ \hbox{and}\ \
\beta_{n+1}^{(2)}(\Sigma,R)\ge\beta_{n+1}^{(2)}(R).
\]
\end{corollary}

\begin{proof}
By the proof of Proposition~\ref{pExact} the sequence
\begin{equation} \label{ereso2}
0\leftarrow M(X)\xleftarrow{\eps}C_0(\Sigma)\xleftarrow{\partial_1}
C_1(\Sigma)\xleftarrow{\partial_2}\dots\xleftarrow{\partial_{n+1}}
C_{n+1}(\Sigma)
\end{equation}
is exact. Taking a projective $\CR$-resolution of
$\ker\partial_{n+1}$ we get a resolution $0\leftarrow M(X)\leftarrow
P_\bullet$ such that its initial segment coincides with
\eqref{ereso2} and $P_k$ is projective for $k\ge n+2$. Then
$$
H_k(\LR\otimes_{\CR}P_\bullet)
=H_k(\LR\otimes_{\CR}C_\bullet(\Sigma))\ \ \hbox{for}\ \ k\le n,
$$
and since $\image(P_{n+2}\to P_{n+1})$ contains the image of
$\partial_{n+2}$, there is a surjective map
$$
H_{n+1}(\LR\otimes_{\CR}C_\bullet(\Sigma))\to
H_{n+1}(\LR\otimes_{\CR}P_\bullet).
$$
This gives the result.
\end{proof}


\begin{thebibliography}{99}

\bibitem{CG}
J. Cheeger, M. Gromov, {\em $L\sb 2$-cohomology and group
cohomology}, Topology {\bf 25} (1986), 189--215.

\bibitem{Ga1}
D. Gaboriau, {\em Invariants $l\sp 2$ de relations d'\'equivalence
et de groupes}, Publ. Math. Inst. Hautes \'Etudes Sci. No. {\bf 95}
(2002), 93--150.

\bibitem{L}
W. L\"uck, {\em $L\sp 2$-invariants: theory and applications to
geometry and $K$-theory}. A Series of Modern Surveys in Mathematics,
{\bf 44}. Springer-Verlag, Berlin, 2002.

\bibitem{S}
R. Sauer, {\em $L\sp 2$-Betti numbers of discrete measured
groupoids}, Internat. J. Algebra Comput. {\bf 15} (2005),
1169--1188.

\bibitem{ST}
R. Sauer, A. Thom, {\em A Hochschild-Serre spectral sequence for
extensions of discrete measured groupoids}, preprint
arXiv:0707.0906v1 [math.DS].

\bibitem{Th}
A. Thom, {\em $L^2$-Invariants and rank metric}, preprint
arXiv:math/0607263v1 [math.OA].

\end{thebibliography}
\end{document}